\documentclass[11pt]{article}
\usepackage{amssymb}
\usepackage{amsmath}
\usepackage[top=1.2in, bottom=1.4in, left=1in, right=1in]{geometry}
\usepackage{parskip}
 \usepackage[T1]{fontenc}
\usepackage{lmodern}
 \usepackage{centernot}
\usepackage{epsfig}
\usepackage{xcolor}
\usepackage{amsmath,amsthm,amsfonts,amssymb,cite,amscd}

\usepackage{setspace}

\usepackage{fullpage}

\usepackage{indentfirst}
\usepackage{mathrsfs}
\usepackage{amsthm}
\usepackage{tikz}
\usepackage{mathtools}
\usepackage{calc}
\usepackage[symbol]{footmisc}
\usetikzlibrary{patterns,arrows,decorations.pathreplacing}
\usetikzlibrary{fadings}
\usepackage{pgfplots}
\usepackage{parskip}
\newtheorem{theorem}{Theorem} 
\newtheorem{lemma}[theorem]{Lemma}
 
\newtheorem{definition}{Definition}
\newtheorem{claim}{Claim}

\newtheorem{conjecture}{Conjecture}

\newtheorem{proposition}{Proposition}

\newcommand{\LL}{\mathcal{L}}

\newcommand{\floor}[1]{\left\lfloor{#1}\right\rfloor}
\newcommand{\ceil}[1]{\left\lceil{#1}\right\rceil}

\newcommand{\RNum}[1]{\uppercase\expandafter{\romannumeral #1\relax}}
\DeclareMathOperator{\ex}{ex}
\DeclareMathOperator{\EX}{EX}

\begin{document}
\title{A note on the Tur\'an number of disjoint union of wheels}

 \author{Chuanqi Xiao\thanks{Central European University, Budapest. email: \texttt{chuanqixm@gmail.com}}\and Oscar Zamora \thanks{Central European University, Budapest. Universidad de Costa Rica, San Jos\'e. email: oscarz93@yahoo.es}
} 
\date{}

\maketitle
\baselineskip=0.3in
\begin{abstract}
      The Tur\'an number of a graph $H$, $\ex(n,H)$, is the maximum number of edges in a
graph on $n$ vertices which does not have $H$ as a subgraph. A wheel $W_n$ is an $n$-vertex graph formed by connecting a single vertex to all vertices of a cycle $C_{n-1}$. Let $mW_{2k+1}$ denote the $m$ vertex-disjoint copies of $W_{2k+1}$. For sufficiently large $n$, we determine the  Tur\'an number and all extremal graphs for $mW_{2k+1}$. We also provide the Tur\'an number and all extremal graphs for $W^{h}:=\bigcup\limits^m_{i=1}W_{k_i}$ when $n$ is sufficiently large, where the number of even wheels is $h$ and $h>0$.
\end{abstract}

\section{Introduction}

In this paper, all graphs considered are undirected, finite and contain neither loops nor multiple edges. Let $G$ be such a graph, the vertex and edge set of $G$ is denoted by $V(G)$ and $E(G)$, the number of vertices and edges in $G$ by $v(G)$ and $e(G)$, respectively. We denote the neighborhood of $v$ in $G$ by $N_G(v)$, the degree of a vertex $v$ in $G$ by $d_G(v)$, the size of $N_G(v)$. Denote by $\chi(G)$ the chromatic number of graph $G$, $\delta(G)$ and $\Delta(G)$ the minimum degree and maximum degree in graph $G$, respectively. Denote by $mH$ the graph of the vertex-disjoint union of $m$ copies of the graph $H$. Two disjoint vertex sets $U$ and $W$ are completely joined in $G$ if $uw\in E(G)$ for all $u\in U$, $w\in W$. Denote by $G_1\bigotimes G_2$ the graph obtained from $G_1 \cup G_2$ and completely join $V(G_1)$ and $V(G_2)$, by $G[B]$ the subgraph of $G$ induced by the vertex set $B$.

The Tur\'an number of a graph $H$, $\ex(n,H)$, is the maximum number of edges in a graph on $n$ vertices which does not have $H$ as a subgraph ($H$-free). Denote by $\EX(n, H)$ the set of $H$-free graphs on $n$ vertices with $\ex(n, H)$ edges and call a graph in $\EX(n, H)$ an extremal graph for $H$.  A \textbf{\textit{wheel}} $W_n$ is a graph on $n$
vertices obtained from a $C_{n-1}$ by adding one vertex $v_0$ and joining $v_0$ to all vertices of the $C_{n-1}$.  We call a wheel on even (odd) vertices even (odd) wheel. In~\cite{DZI}, Dzido determined for $k\geq 3$ and $n\geq 6k-10$, $\ex(n,W_{2k})=\left\lfloor\frac{n^3}{3}\right\rfloor$. Yuan~\cite{YUAN} proved the Tur\'an number $\ex(n,W_{2k+1})$ for odd wheel when $n$ is sufficiently large. Motivated by these results we determin the Tur\'an number and characterize all extremal graphs for disjoint union of wheels. 

Denote by $P_n$ the path on $n$ vertices and $K_{a,b}$  the complete bipartite graph with $a$ and $b$ vertices in its color classes.  
Let $\mathcal{U}^{k-1}_{n}(P_{2k-1})$  be the class of $P_{2k-1}$-free, $(k-1)$-regular or nearly $(k-1)$-regular graphs on $n$ vertices. 

\begin{definition}
Let $\mathcal{K}^{t}_{n_1,n_2}\left(\mathcal{U}^{k-1}_{n_1}(P_{2k-1});P_2\right)=K_t\bigotimes \mathcal{K}_{n_1,n_2}\left(\mathcal{U}^{k-1}_{n_1}(P_{2k-1});P_2\right)$, $n_1\geq n_2\geq 2$ and $n_1+n_2=n-t$, where $\mathcal{K}_{n_1,n_2}\left(\mathcal{U}^{k-1}_{n_1}(P_{2k-1});P_2\right)$ denote the class of graphs obtained from a $K_{n_1,n_2}$ by embedding the larger partite set a graph from $\mathcal{U}^{k-1}_{n_1}(P_{2k-1})$ and  embedding the smaller partite set an
edge. 
\end{definition}
\begin{theorem}\label{main1}
Let $mW_{2k+1}$ denote the $m$ vertex-disjoint copies of $W_{2k+1}$. For $n$ sufficiently large, $$\ex(n,mW_{2k+1})=\max\left\{{m-1 \choose 2}+ \floor{\frac{(k-1)n_0}{2}}+(n_0+m-1)(n-m+1)-n^2_0+1\right\}$$ 
and $\EX(n, mW_{2k+1})\subseteq \mathcal{K}^{m-1}_{n_1,n_2}\left(\mathcal{U}^{k-1}_{n_1}(P_{2k-1});P_2\right)$. 
\end{theorem}

\section{Preliminary}
  Clearly, the graphs in  $\mathcal{K}^{m-1}_{n_1,n_2}\left(\mathcal{U}^{k-1}_{n_1}(P_{2k-1});P_2\right)$ are $mW_{2k+1}$-free. In~\cite{YUAN}, Yuan showed the case when $m=1$.


 \begin{theorem}[Yuan,~\cite{YUAN}]\label{Y}
 Let $k\geq 2$ and $W_{2k+1}$ be a wheel on $2k+1$ vertices. Then for $n$ sufficiently large,
 \begin{align*}
 \ex(n,W_{2k+1})=
 \begin{cases}
 &\left(\left\lceil\frac{n}{2}\right\rceil+1\right)\left\lfloor\frac{n}{2}\right\rfloor, k=2,\\
&\max\left\{n_on_1+\left\lfloor\frac{(k-1)n_0}{2}\right\rfloor:n_0+n_1=n\right\}+1, k\geq3,
 \end{cases}
 \end{align*}
 and $\EX(n,W_{2k+1})\subseteq \mathcal{K}^{0}_{n_1,n_2}\left(\mathcal{U}^{k-1}_{n_1}(P_{2k-1});P_2\right)$.
 \end{theorem}

To prove Theorem~\ref{main1}, we use the technique of progressive induction. Essentially, the technique is as follows. For a given problem you are able to prove the inductive step
under the assumptions of the inductive hypothesis. 
However, you are unable to prove the anchor step (this could be because the anchor step is not true for small values). 
It also appears that the proof of the anchor step is as difficult as a direct proof of the result. 
Formally the statement we use is the following:

\begin{proposition} Let $c\in \mathbb{N}$ and $\varphi:\mathbb{N} \to \mathbb{Z}$ be a function such that  $\varphi(n) < \max\{\varphi(n-1),\varphi(n-c)\}$, then there exists $n_0 \in \mathbb{N}$ such that $\varphi(n) < 0$ for every $n>n_0.$
\end{proposition}

Let $H_{n}$ be an extremal graph for $mW_{2k+1}$ and $f(n,t)=\max\left\{e(G):G\in \mathcal{K}^{t}_{n_1,n_2}\left(\mathcal{U}^{k-1}_{n_1}(P_{2k-1});P_2\right)\right\}$. To establish the result, in this paper, we define a function $\varphi(n)$, used to measure the ``distance
between our knowledge $e(H_n)$ and the conjecture $f(n,m-1)$'', that is $\varphi(n)=e(H_n)-f(n,m-1)$. Clearly, $\varphi(n)$ is non-negative, we then attempt to show that there exists $n_0$, when $n>n_0$, either $\varphi(n)<\varphi(n-1)$, $\varphi(n)<\varphi(n-c)$ (for some $c$ chosen later) or $H_n\in \mathcal{K}^{m-1}_{n_1,n_2}\left(\mathcal{U}^{k-1}_{n_1}(P_{2k-1});P_2\right)$.

\section{Disjoint copies of odd wheels}

We need the following theorem and key lemma to proof Theorem~\ref{main1}.
 \begin{theorem}[K\H ov\'ari-S\'os-Tur\'an,\cite{KST}]\label{KS}
 Let $K_{a,b}$ denote the complete
bipartite graph with $a$ and $b$ vertices in its color-classes. Then
$$\ex(n,K_{a,b})\leq \frac{\sqrt[a]{b-1}}{2}n^{2-\frac{1}{a}}+\frac{a-1}{2}n.$$
 \end{theorem}

\begin{lemma} \label{mainl}

Let $G$ be an $mW_{2k+1}$-free graph with a partition of the vertices into two nonempty parts $V(G) = V_1 \cup V_2$ with sizes $n_1$ and $n_2$ respectively such that $n_1 \geq n_2$ and $n_2$ is sufficiently large. 
Suppose $G$ is such, for each $i$ if $S \subseteq V_i$ has size at most $m(k+1)$ then all vertices in $S$ have at least $m(2k+1)$ common neighbors in the other class. 
Then, for $n_1$ sufficiently large, 
$e(G) \leq g(n_1,n_2,m),$ where $g(n_1,n_2,m)$ is defined as 
$$g(n_1,n_2,m) =\max\left\{e\left(\mathcal{K}^{m-1}_{n_1-j,n_2-(m-1-j)}(\mathcal{U}^{k-1}_{n_1-j}(P_{2k-1});P_2)\right):j=0,1,\dots,m-1\right\}.$$
Moreover, for $m>1$ equality can only hold if $G$ contains a vertex of degree $n_1+n_2-1$.
\end{lemma}

\begin{proof}

The proof will follow by induction on $m$, the case where $m=1$ is done by~\cite{YUAN}.

Clearly,  for $n_2 \geq m-1$ we have that $g(n_1,n_2,m) \leq f(n,m-1)$.
Now suppose that $m>1$, note that by the definition of $\mathcal{K}^{t}_{n_1,n_2}\left(\mathcal{U}^{k-1}_{n_1}(P_{2k-1});P_2\right)$  we have $$e\left(\mathcal{K}^{m-1}_{n_1-j,n_2-(m-1-j)}(\mathcal{U}^{k-1}_{n_1-j}(P_{2k-1});P_2)\right) = e\left(\mathcal{K}^{m-2}_{n_1-j,n_2-(m-1-j)}(\mathcal{U}^{k-1}_{n_1-j}(P_{2k-1});P_2)\right) + (n_1+n_2-1).$$

It follows from the definition that both $g(n_1-1,n_2,m-1)$ and $g(n_1,n_2-1,m-1)$ are bounded above by $g(n_1,n_2,m)-(n_1+n_2-1).$

  
    

Let $S_{n}$ denote the star on $n$ vertices and $G_i$ denote the subgraph of $G$ induced by the vertex set $V_i$. For a graph $H$, let $s_{k+1}(H)$ denote the maximum number of disjoint $S_{k+1}$ in $H$. From the conditions of $G$ we have that $s_{k+1}(G_1) + s_{k+1}(G_2) \leq m-1.$ 
We separate the proof into $2$ cases.

\textbf{Case 1}.\ \ For some $i$ there exists a vertex $u \in V_i$ such that $d_{G_i}(u) \geq m(2k+1)$.

Let $G'$ be the graph obtain from $G$ by removing $u$, then the vertex set of $G'$  can be decomposed into graphs $V_1' \cup V_2'$ of sizes $n_1'$ and $n_2'$. We have that $G'$ must be $(m-1)W_{2k+1}$-free, otherwise we may find another wheel with center $u$ which is disjoint from the previous $(m-1)W_{2k+1}$. Hence, by the induction hypothesis we have $e(G') \leq g(n_1',n_2',m-1)$ and so $$e(G) \leq d_G(u) + g(n_1',n_2',m-1) \leq n_1+n_2-1 + g(n_1',n_2',m-1) \leq g(n_1,n_2,m-1),$$ 
 the equality holds only when $d_G(u) = n_1+n_2-1$.

\textbf{Case 2}.\ \ For each vertex $v \in V_i$ ($i=1,2$), $d_{G_i}(v) < m(2k+1)$.

Then we have that $d(v) < n_2 + m(2k+1)$ for $v \in V_1$ while $d(v) < n_1+m(2k+1)$ for $v \in V_2$. We may assume by induction that $G$ contains at least one wheel $W$, say with vertices $a_1,a_2,\dots,a_s$ in $V_1$ and $b_1,\dots,b_{t}$ in $V_2$, where $s+t = 2k+1$. Then $G'$, defined as the graph obtain by $G$ by removing $W$, can be decomposed in components $V_1'$ and $V_2'$ of sizes $n_1-s, n_2-t$ respectively, then \begin{equation} \label{gap}
e(G) \leq sn_2 + tn_1 + (2k+1)^2m + g(n_1-s,n_2-t,m-1). 
\end{equation}

Note that by the construction of $G$ we have the following bounds 
\begin{align*}
g(x,y,m)& \geq g(x,y,m-1) + \min\{y,x-k\}-m \geq  g(x,y,m-1) + y-k-m \\
g(x,y,m)& \geq g(x-1,y,m) +y\\
g(x,y,m)& \geq g(x,y-1,m) + x    
\end{align*}

The first bound is obtained by the difference between the number of edges of the graphs in the definition of $g$, that is comparing the number of edges of $\mathcal{K}^{m-1}_{n_1-j,n_2-(m-1-j)}\left(\mathcal{U}^{k-1}_{n_1-j}(P_{2k-1});P_2\right)$
with $\mathcal{K}^{m-2}_{n_1-j+1,n_2-(m-1-j)}\left(\mathcal{U}^{k-1}_{n_1-j}(P_{2k-1});P_2\right)$ (when $j\geq 1$) or $\mathcal{K}^{m-2}_{n_1-j,n_2+1-(m-1-j)}\left(\mathcal{U}^{k-1}_{n_1-j}(P_{2k-1});P_2\right)$ (when $j\leq m-2$).

As a consequence of these bounds it follows that 
\begin{align*}
g(n_1-s,n_2-t,m-1) \leq g(n_1,n_2,m) - sn_2 - tn_1 - (n_2-k-m) + (2k+1)^2
    \end{align*}
Hence together with equation \eqref{gap} it follow that $$e(G) \leq  g(n_1,n_2,m) +  (2k+1)^2(m+1) - (n_2-k-m),$$ 
then $e(G)<g(n_1,n_2,m)$ if \begin{displaymath} n_2 > (2k+1)^2(m+1) +m+k.\qedhere \end{displaymath}
    \end{proof}

\begin{lemma}[Yuan, \cite{YUAN}]\label{Y1}
Let $n\geq 2k$, then $\ex(n,\{S_{k+1},P_{2k+1}\})=\floor{\frac{(k-1)n}{2}}$.
\end{lemma}

\begin{proof}[Proof of Theorem~\ref{main1}]
We proof Theorem~\ref{main1} using the progressive induction.
Let $n$ be large enough and $H_n$ be an $n$-vertex $mW_{2k+1}$-free graph with maximal number of edges. We will also assume by induction that Theorem~\ref{main1} holds for $m-1$, the base case $m=1$ is done by~\cite{YUAN}. The following proof is based on Yuan's result.

Since  $e(H_n) >\floor{\frac{n^2}{4}}$, by Theorem~\ref{KS}, there exists $n_1$ such that when $n>n_1$, $H_n$ contains $K_{N,N}$ as a subgraph, for some large and even $N$.
Let $B_1$ and $B_2$ be the bipartite classes of $K_{N,N}$.
Let $\hat{H}_{2N}$ be the graph induced by the vertex set $B_1\cup B_2$, $\Tilde{H}_{n-2N}$ be the graph induced by the vertex set $V(H_n)\setminus (B_1\cup B_2)$ and $e_{H}$ be the number of edges between $\hat{H}_{2N}$ and $\Tilde{H}_{n-2N}$. 
Thus, $$e(H_n)=e(\hat{H}_{2N})+e_{H}+e(\Tilde{H}_{n-2N}).$$

Let $H^{'}_{n}$ be a graph in $\mathcal{K}^{m-1}_{n_1,n_2}\left(\mathcal{U}^{k-1}_{n_1}(P_{2k-1});P_2\right)$, by Lemma \ref{Y1}, there exists a graph $H^{'}_{n}$ such that $K^*_{N,N} \subseteq H^{'}_n$, for some $K^*_{N,N} \in \mathcal{K}_{N,N}\left(\mathcal{U}^{k-1}_{N}(P_{2k-1});\emptyset\right).$ 
Let $H^{'}_{n-2N}$ be the graph induced by the vertex set $V(H^{'}_{n})\setminus V\left(K^*_{N,N}\right)$ and $e_{H^{'}}$  be the number of edges joining $K^*_{N,N}$ and $H^{'}_{n-2N}$. 
Thus, $$e(H^{'}_{n})=e\left(K^*_{N,N}\right)+e_{H^{'}}+e(H^{'}_{n-2N}).$$
Clearly, $e_{H'}=(n-2N)N+(m-1)N=(n-2N+m-1)N$. 

By Lemma~\ref{mainl}, we see $e(\hat{H}_{2N}) \leq g(N,N,m).$ 
Therefore, we have 
\begin{align}\label{n-2N}
    \varphi(n)
    &=e(H_{n})-e(H^{'}_{n}) \nonumber\\
    &=e(\hat{H}_{2N})-e\left(K^*_{N,N}\right)+e_{H}-e_{H^{'}}+e(\Tilde{H}_{n-2N})-e(H^{'}_{n-2N})\nonumber\\ 
    &\leq g(N,N,m) -N^2 - \frac{N(k-1)}{2} +(e_{H}-e_{H^{'}})+\varphi(n-2N)\nonumber\\ 
    &\leq mN +(e_{H}-e_{H^{'}})+\varphi(n-2N)
\end{align}

Note that from \eqref{n-2N} we have that if $\varphi(n) \geq \varphi(n-2N)$ then $mN \geq e_{H'} - e_H$.

To complete the progressive induction, we are going to show that for $n$ large enough, either $\varphi(n)<\varphi(n-2N)$ or $\varphi(n)<\varphi(n-1)$ or $H_n\in \mathcal{K}^{m-1}_{n_1,n_2}\left(\mathcal{U}^{k-1}_{n_1}(P_{2k-1});P_2\right)$.

\textbf{Case 1}.\ \ There exists a vertex $v\in H_n$ with $d_{H_{n}}(v)<\frac{n}{2}$.

Since $e(H^{'}_n)=f(n,k-1)=\max\left\{{m-1 \choose 2}+ \floor{\frac{(k-1)n_0}{2}}+(n_0+m-1)(n-m+1)-n^2_0+1\right\}$ where $n_0=\frac{1}{2}\left(\floor{\frac{k-1}{2}}+n-m+1\right)$ or $n_0=\frac{1}{2}\left(\ceil{\frac{k-1}{2}}+n-m+1\right)$, we get $e(H^{'}_n)-e(H^{'}_{n-1})=f(n,k-1)-f(n-1,k-1)\geq \frac{n}{2}$. Clearly, $H_n-v$ is an $(n-1)$-vertex $mW_{2k+1}$-free graph which implies that $e(H_n)-d_{H_{n}}(v)\leq e(H_{n-1})$. 
Hence, $e(H_n)- e(H_{n-1})\leq d_{H_{n}}(v)<\frac{n}{2}$ and we get $\varphi(n)=e(H_n)-e(H^{'}_{n})<e(H_{n-1})-e(H^{'}_{n-1})=\varphi(n-1)$.

In Case $2$ we will assume that neither $\varphi(n)<\varphi(n-2N)$ nor $\varphi(n)<\varphi(n-1)$ hold.

\textbf{Case 2}.\ \  $\delta(H_n)\geq \frac{n}{2}$ and  $\varphi(n)\geq\varphi(n-2N)$. With the following claims we are able to show that $H_n\in \mathcal{K}^{m-1}_{n_1,n_2}\left(\mathcal{U}^{k-1}_{n_1}(P_{2k-1});P_2\right)$ in this case.


\begin{claim}\label{bigdegree} Let $x$ be a vertex in $H_n$ such that $K_{m(2k+1),m(2k+1)}$ is contained in the neighborhood of $x$, then $G'$, the graph induced by $V(H_n) \setminus\{v\}$, is $(m-1)W_{2k+1}$-free.
\end{claim}
\begin{proof}

Suppose by contradiction that $G'$ is not $(m-1)W_{2k+1}$-free, since a copy of $(m-1)W_{2k+1}$ contains $m(2k+1)$ vertices in $G'$, then we may find a copy of $K_{k,k}$ in the neighborhood of $x$ which does not contain any vertex of the given $(m-1)W_{2k+1}$ copy, then $v$ together with the copy of $K_{k,k}$ contains another copy of $W_{2k+1}$ which contradicts the fact that $H_n$ is $mW_{2k+1}$-free. 
\end{proof}


Hence we may assume that for any vertex  $v\in V(H_n)$, there is an index $i(v) \in \{1,2\}$ such that $v$ has less than $m(2k+1)$ neighbors in $B_{i(v)}$, since otherwise we would be able to find a copy of $K_{m(2k+1),m(2k+1)}$ in the neighborhood of $v$, and then by Claim~\ref{bigdegree} and induction on $m$, we would have that $$e(H_n) \leq (n-1) + e(G[V(H_n)\setminus\{v\}]) \leq (n-1) + f(n-1,m-2) = f(n,m-1).$$
 where the equality holds only if $d_{H_n}(v) = n-1$ and the graph induce by $V(H_n)\setminus\{v\}$ is in $\mathcal{K}^{m-2}_{n_1',n_2'}\left(\mathcal{U}^{k-1}_{n_1}(P_{2k-1});P_2\right)$ for some $n_1'+n_2' = n-1$. Therefore, by adding a full degree vertex to the previous graph with have that the equality holds only when $H_n  \in \mathcal{K}^{m-1}_{n_1,n_2}\left(\mathcal{U}^{k-1}_{n_1}(P_{2k-1});P_2\right)$ for some $n_1$ and $n_2$ with $n_1 + n_2 = n$ which maximizing the number of edges.

We partition the vertices of $\Tilde{H}_{n-2N}$ into the following classes: $C_1$, $C_2$ and $D$ such that: $ C_i$ is the set of vertices $v$ such that $v$ is adjacent to less than $m(2k+1)$ vertices in $B_i$ and more than $N - 2m(2k+1)$ vertices of $B_{3-i}$ for $i=1,2$, $v\in D$ if $v$ is adjacent to at most $N- 2m(2k+1)$ vertices of both $B_1$ and $B_2$.

By the definition of $C_i$ we have that any $m(2k+1)+1$ vertices of $B_i \cup C_i$ have more than $N   2m(2k+1)\bigg(m(2k+1)+1\bigg) \geq m(2k+1)$ neighbors in $B_{3-i}$, hence we may assume that every vertex $x \in B_i \cup C_i$ has less than $m(2k+1)$ neighbors in $B_i \cup C_i$ or we would be done by Claim~\ref{bigdegree}.




\begin{claim}
There exists a constant $N_1$ such that $|D|<N_1$.
\end{claim}
\begin{proof}

Recall that by definition every vertex in $C_i$ is adjacent to less than $m(2k+1)$ vertices of $B_i$ and 
for each vertex $v\in D$, there exists an $i(v)$ such that $v$ is join to less than $m(2k+1)$ vertices of $B_{i(v)}$, we get that $v$ is join to less than $m(2k+1)+(N - 2m(2k+1))\leq N-m(2k+1)$ vertices of $\hat{H}_{2N}$. 
Therefore,
\begin{align*}
    e_H
    &=e(B_1,C_1)+e(B_2,C_2)+e(B_1,C_2)+e(B_2,C_1)+e(B_1\cup B_2,D)\\
    &\leq 2Nm(2k+1)+N(n-2N)-m(2k+1)|D|\\
    &=4Nmk+N(m+1)+N(n-2N+m-1)-m(2k+1)|D|
\end{align*}
Since $e_{H'}= N(n-2N+m-1)$, we have that $e_H\leq 4Nmk+N(m+1)+e_{H^{'}}-m(2k+1)|D|$. From inequality \eqref{n-2N} we have $$mN \geq e_{H^{'}}-e_{H}\geq mk|D| - 4Nmk-N(m+1),$$ 
hence $|D|< N\frac{4k+3}{k}=N_1$.
\end{proof}
\begin{claim}
$|B_i\cup C_i|=\frac{n}{2}+O(\sqrt{n})$.
\end{claim}

\begin{proof}
Since there exists an integer $N_1$ such that $|D|\leq N_1$, then the number of edges incidence with $D$ is $O(n)$. 
Since $\Delta(G[B_i\cup C_i])< m(2k+1)$, we see $e(G[B_1\cup C_1])+e(G[B_2\cup C_2])=O(n)$. 
Hence, after removing the edges in $G[B_1\cup C_1]$, $G[B_2\cup C_2]$ and the edges incidence with $D$, we obtain a bipartite graph on $\floor{\frac{n^2}{4}}-O(n)$ edges. 
Therefore, there exists a constant $N_2$ such that $\left||B_i\cup C_i|-\frac{n}{2}\right|\leq N_2\sqrt{n}$, hence, $|B_i\cup C_i|=\frac{n}{2}+O(\sqrt{n})$.
\end{proof}





\begin{claim}\label{Dbigdegree} 
$D=D_1\cup D_2$, where vertices in $D_i$ is adjacent to less than $m(2k+1)$ vertices of $B_i\cup C_i$.
\end{claim}

\begin{proof}
Let $v\in D$, then there exists an $j(v)$ such that $v$ is adjacent to at least $\frac{n}{6}$ vertices in $B_{j(v)}\cup C_{j(v)}$. 
Otherwise, $d_{H_{n}}(v)<N_1-1+2\frac{n}{6}<\frac{n}{2}$, which contradicts to the fact that $\delta(H_n)\geq \frac{n}{2}$. 
Hence, since each vertex $u\in B_i\cup C_i$ has more than $\frac{n}{2}-O(n)$ neighbors in $B_{3-i} \cup C_{3-i}$, if a vertex $v_0\in D$ is adjacent to at least $m(2k+1)$ vertices in $B_{3-j(v_0)} \cup C_{3-j(v_0)}$ we may find a copy of $K_{m(2k+1),m(2k+1)}$ and we would be able to apply Claim~\ref{bigdegree}.  Let $D_i\subseteq D$, such that each vertex $v\in D_i$ is adjacent to less than $m(2k+1)$ vertices in $B_i\cup C_i$, then $D$ is the disjoint union of $D_1$ and $D_2$.
\end{proof}


Hence we may assume that every vertex $x \in D$ has less than $m(2k+1)$ neighbors in one of the classes $B_1 \cup C_1$ or $B_2 \cup C_2$, otherwise we would be done by induction.

 

Let $V_1=B_1\cup C_1\cup D_1$ and $V_2=B_2\cup C_2\cup D_2$, then $V_1$ and $V_2$ is a vertex partition of $H_n$ such that for any vertex set on $m(k+1)$ vertices in $V_i$ has at least $m(2k+1)$ common neighbors in $V_{3-i}$. 
Then by Lemma~\ref{mainl}, we get $e(H_n)\leq f(n,m-1)$, the equality holds only when $H_n$ contains a vertex $v$ of degree $n-1$. Therefore, $v$ would have at least $m(2k+1)$ neighbors in both $B_1$ and $B_2$, which is a contradiction. 
\end{proof}

\section{Remarks and Open Problems}

We now considering the disjoint union of wheels of different sizes.
When there is an even wheel the following result holds.

The Tur\'an graph $T(n,p)$ is a complete multipartite graph formed by partitioning a set of $n$ vertices into $p$ subsets, with sizes as equal as possible, and connecting two vertices by an edge if and only if they belong to different subsets. Denote its size by $t(n,p)$.

\begin{theorem}\label{main2}
Let $W^h=\bigcup\limits^m_{i=1}W_{k_i}$ be a disjoint union of $m$ wheels and the number of even wheels is $h$, $(h\geq 1)$, then  for $n$ sufficiently large, $\ex(n,W^h)=\left\{{h-1 \choose 2}+ (h-1)(n-h+1)+t(n-h+1,3)\right\}$ and $\EX(n,W^{h})=K_{h-1}\bigotimes T(n-h+1,3)$.
\end{theorem}

Theorem~\ref{main2} is a consequence of the following result of Simonovits~\cite{SIM}

\begin{theorem}[\textbf{Simonovits}\cite{SIM}]\label{SIM1}
Let $\LL$ be the family of forbidden graphs and $p=p(\LL) = \min\limits_{L\in\LL}\chi(\LL)-1$.  
If by omitting any $s-1$ vertices of any $L\in \LL$ we obtain a graph with chromatic number at least $p+1$,  but  by  omitting $s$ suitable  edges  of  some $L\in \LL$ we  get  a $p$-colorable graph, then $K_{s-1} \bigotimes T(n-s+1,p)$ is the unique extremal graph for $\LL$ when $n$ is sufficiently large.
\end{theorem}

Let $k_1\geq  k_2\geq \dots\geq k_m$ be positive integers, it is easy to see that if the disjoint union of stars $\bigcup\limits^{m}_{i=1}S_{k_i+1}$,  is added to one class of $K_{n_0,n_1}$, the we would obtain a copy of $\bigcup\limits^m_{i=1}W_{2k_i+1}$. 

Base on  the following theorem, we propose a conjecture on the extremal number for $\bigcup\limits^m_{i=1}W_{2k_i+1}$. 

\begin{theorem}[Lidick\'y, Liu, Palmer \cite{LIU}]\label{LIU} 
Let $F=\bigcup\limits^{k}_{i=1}S^{i}$ be a star forest where $d_i$ is the maximum degree of $S^{i}$ and $d_1\geq d_2\geq\dots\geq d_k$. For $n$ sufficiently large, 
$$\ex(n,F)=\max\limits_{1\leq i\leq k}\left\{(i-1)(n-i-1)+{i-1\choose2}+\floor{\frac{d_i-1}{2}(n-i-1)}\right\}.$$
\end{theorem}

\begin{conjecture}
Let $\bigcup\limits^m_{i=1}W_{2k_i+1}$ be a disjoint union of odd wheels with components order $2k_1+1, 2k_2+1, \dots, 2k_m+1$ where $k_1\geq k_2\geq \dots\geq k_m$. For $n$ sufficiently large, 
\begin{align*}
    &\ex(n,\bigcup\limits^m_{i=1}W_{2k_i+1})\\
    &=\max_{1\leq n_0 \leq n}\left\{n_0(n-n_0)+\ex(n_0,\bigcup\limits^{m}_{i=1}S_{k_i+1})+1\right\}\\
   &=\max_{\substack{1\leq i\leq m \\ 1\leq n_0 \leq n} }\left\{n_0(n-n_0)+(i-1)(n_0-i-1)+{i-1\choose 2}+\floor{\frac{k_i-1}{2}(n_0-i-1)}+1\right\}
\end{align*}
\end{conjecture}

\end{document}